\documentclass[11pt,a4paper]{amsart}
\usepackage[arrow,curve,matrix]{xy}
\usepackage[a4paper,twoside,centering]{geometry}
\usepackage{amssymb}

\title{Which mutation classes of quivers have constant number of arrows?}

\author{Sefi Ladkani}
\address{%
Institut des Hautes \'{E}tudes Scientifiques \\
Le Bois Marie, 35, route de Chartres \\
91440 Bures-sur-Yvette, France}
\email{sefil@ihes.fr}
\urladdr{http://www.ihes.fr/\~{}sefil}

\thanks{This work was supported by a European Postdoctoral Institute (EPDI)
fellowship.}

\newcommand{\cP}{\mathcal{P}}
\newcommand{\cQ}{\mathcal{Q}}

\theoremstyle{plain}
\newtheorem{theorem}{Theorem}
\newtheorem{lemma}{Lemma}[section]
\newtheorem{prop}[lemma]{Proposition}
\newtheorem*{prop*}{Proposition}
\newtheorem{cor}[lemma]{Corollary}
\newtheorem*{cor*}{Corollary}

\theoremstyle{definition}
\newtheorem*{defn*}{Definition}

\newtheorem*{quest*}{Question}
\newtheorem{remark}[lemma]{Remark}

\numberwithin{equation}{section}

\begin{document}

\begin{abstract}
We classify the connected quivers with the property that all the
quivers in their mutation class have the same number of arrows. These
are the ones having at most two vertices, or the ones arising from
triangulations of marked bordered oriented surfaces of two kinds:
either surfaces with non-empty boundary having exactly one marked point
on each boundary component and no punctures, or surfaces without
boundary having exactly one puncture.

This combinatorial property has also a representation-theoretic
counterpart: to each such quiver there is a naturally associated
potential such that the Jacobian algebras of all the QP in its mutation
class are derived equivalent.
\end{abstract}

\maketitle

\section{Motivation and Summary of Results}

Quiver mutation is a combinatorial notion introduced by Fomin and
Zelevinsky~\cite{FominZelevinsky02} in their theory of cluster
algebras. Let us briefly recall its definition.

A \emph{quiver} is a directed graph, where multiple arrows between two
vertices are allowed. Throughout this paper, we consider only quivers
without \emph{loops} (arrows starting and ending at the same vertex)
and \emph{$2$-cycles} (i.e.\ pairs of arrows $i \to j$ and $j \to i$).
Let $Q$ be such quiver and let $k$ be a vertex of $Q$. The \emph{quiver
mutation} of $Q$ at $k$ is a new quiver $\mu_k(Q)$ obtained from $Q$ by
performing the following three steps:
\begin{enumerate}
\renewcommand{\theenumi}{\roman{enumi}}
\item
For any pair of arrows of the form $(i \to k, k \to j)$, add a new
arrow $i \to j$;

\item \label{it:mut:rev}
Reverse the direction of all the arrows starting or terminating at $k$;

\item
Remove a maximal set of $2$-cycles.
\end{enumerate}

When no arrow starts at the vertex $k$ (it is then called a
\emph{sink}) or no arrow ends at $k$ (it is then called a
\emph{source}), mutation at $k$ reduces to step~(\ref{it:mut:rev})
above and coincides with the \emph{BGP reflection} considered by
Bernstein, Gelfand and Ponomarev~\cite{BGP73}. Obviously, in this case
the quivers $Q$ and $\mu_k(Q)$ have the same number of arrows.

However, mutation in general does not preserve the number of arrows in
the quivers, as can be seen for example by mutating the left quiver
with two arrows at the vertex $2$ obtaining the right quiver with three
arrows.
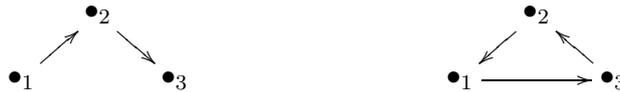
\begin{align*}
\xymatrix@=1pc{
& {\bullet_2} \ar[dr] \\
{\bullet_1} \ar[ur] & & {\bullet_3}
}
& &
\xymatrix@=1pc{
& {\bullet_2} \ar[dl] \\
{\bullet_1} \ar[rr] & & {\bullet_3} \ar[ul]
}
\end{align*}

It is therefore interesting to search for quivers with the property
that performing an arbitrary sequence of mutations does not change the
number of arrows. To formulate this more precisely, recall that the
\emph{mutation class} of a quiver $Q$ consists of all the quivers that
can be obtained from $Q$ by performing sequences of mutations.

\begin{quest*}
Which mutation classes have the property that all their quivers have
the same number of arrows?
\end{quest*}

\subsection{Combinatorial results}

It turns out that the answer to this question is closely related with
the analysis of mutation classes of quivers arising from triangulations
of marked bordered oriented surfaces as introduced by Fomin, Shapiro
and Thurston~\cite{FST08}. Recall that these consist of pairs $(S,M)$
where $S$ is a compact connected oriented Riemann surface and $M
\subset S$ is a finite set of \emph{marked points} containing at least
one point from each connected component of the boundary of $S$ (which
might be empty). The homeomorphism type of $(S,M)$ is governed by the
following discrete data:
\begin{itemize}
\item
the \emph{genus} $g \geq 0$ of the surface $S$;

\item
the number $b \geq 0$ of connected components of its boundary;

\item
the number of marked points on each boundary component;

\item
the number of marked points not on the boundary (called
\emph{punctures}).
\end{itemize}

The following mutation classes will play significant role in our
considerations.
\begin{defn*}
Let $g, b \geq 0$ such that $(g,b) \not \in \{(0,0), (0,1)\}$.
\begin{enumerate}
\renewcommand{\theenumi}{\alph{enumi}}
\item
If $b=0$, denote by $\cQ_{g,0}$ the mutation class consisting of the
quivers arising from triangulations of a surface without boundary of
genus $g$ with one puncture.

\item
If $b>0$, denote by $\cQ_{g,b}$ the
mutation class consisting of the quivers arising from triangulations
of a surface of genus $g$ with $b$ boundary components and exactly
one marked point on each boundary component.
\end{enumerate}
\end{defn*}

A procedure to produce explicit members from the classes $\cQ_{g,b}$
when $b>0$ has been described in our previous work~\cite{Ladkani11}.
With slight modifications, it can also be used to produce explicit
members from the classes $\cQ_{g,0}$, see Section~\ref{sec:Qg0}.
Examples of such quivers for small values of $g$ and $b$ are shown in
Figure~\ref{fig:Qg0} and Figure~\ref{fig:Qgb}.

\begin{figure}
\[
\begin{array}{cc}
(1,0) &
\begin{array}{c}
\xymatrix@=0.5pc{
& {\bullet} \ar@/^0.1pc/[ddr] \ar@/_0.1pc/[ddr] \\ \\
{\bullet} \ar@/^0.1pc/[uur] \ar@/_0.1pc/[uur]
& & {\bullet} \ar@/^0.1pc/[ll] \ar@/_0.1pc/[ll]
}
\end{array}
\\ \\
(2,0) &
\begin{array}{c}
\xymatrix@=1pc{
& & {\bullet} \ar[dl] \ar[dd] & & {\bullet} \ar[dd] \ar[drr] \\
{\bullet} \ar[urr] \ar[drr] & {\bullet} \ar@/^0.1pc/[l] \ar@/_0.1pc/[l]
& & {\bullet} \ar[ul] \ar[ur]
& & {\bullet} \ar[ul] \ar[dl] & {\bullet} \ar@/^0.1pc/[l] \ar@/_0.1pc/[l] \\
& & {\bullet} \ar[ul] \ar[ur] & & {\bullet} \ar[ul] \ar[urr]
}
\end{array}
\\ \\
(3,0) &
\begin{array}{c}
\xymatrix@=1pc{
& & {\bullet} \ar[dl] \ar[dd]
& & {\bullet} \ar[rr] \ar[dd] & & {\bullet} \ar[dl] \ar[dd]
& & {\bullet} \ar[dd] \ar[drr] \\
{\bullet} \ar[urr] \ar[drr] & {\bullet} \ar@/^0.1pc/[l] \ar@/_0.1pc/[l]
& & {\bullet} \ar[ul] \ar[ur]
& & {\bullet} \ar[dl] \ar[ul]
& & {\bullet} \ar[ul] \ar[ur] & & {\bullet} \ar[ul] \ar[dl]
& {\bullet} \ar@/^0.1pc/[l] \ar@/_0.1pc/[l] \\
& & {\bullet} \ar[ul] \ar[ur]
& & {\bullet} \ar[ul] \ar[rr] & & {\bullet} \ar[ul] \ar[ur]
& & {\bullet} \ar[ul] \ar[urr]
}
\end{array}
\\ \\
(4,0) &
\begin{array}{c}
\xymatrix@=1pc{
& & {\bullet} \ar[dl] \ar[dd]
& & {\bullet} \ar[rr] \ar[dd] & & {\bullet} \ar[dl] \ar[dd]
& & {\bullet} \ar[rr] \ar[dd] & & {\bullet} \ar[dl] \ar[dd]
& & {\bullet} \ar[dd] \ar[drr] \\
{\bullet} \ar[urr] \ar[drr] & {\bullet} \ar@/^0.1pc/[l] \ar@/_0.1pc/[l]
& & {\bullet} \ar[ul] \ar[ur] & & {\bullet} \ar[ul] \ar[dl]
& & {\bullet} \ar[ul] \ar[ur]
& & {\bullet} \ar[ul] \ar[dl] & & {\bullet} \ar[ul] \ar[ur]
& & {\bullet} \ar[ul] \ar[dl] & {\bullet} \ar@/^0.1pc/[l] \ar@/_0.1pc/[l] \\
& & {\bullet} \ar[ul] \ar[ur]
& & {\bullet} \ar[ul] \ar[rr] & & {\bullet} \ar[ul] \ar[ur]
& & {\bullet} \ar[ul] \ar[rr] & & {\bullet} \ar[ul] \ar[ur]
& & {\bullet} \ar[ul] \ar[urr]
}
\end{array}
\end{array}
\]
\caption{Representative quivers in $\cQ_{g,0}$ for $g=1,2,3,4$.}
\label{fig:Qg0}
\end{figure}
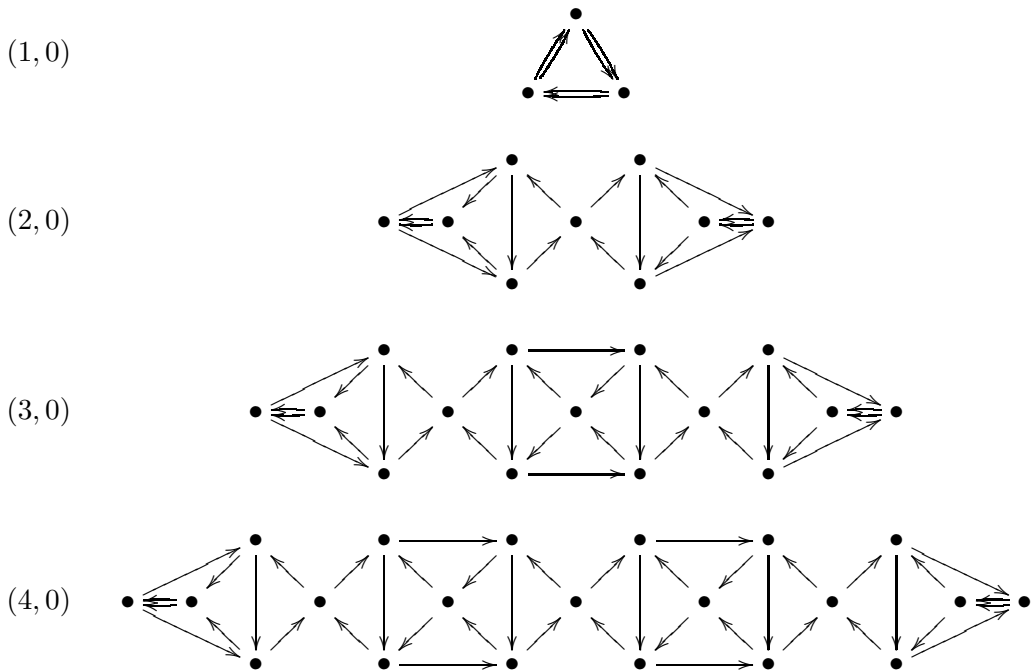

\begin{figure}
\[
\begin{array}{ccc}
\begin{array}{c}
\xymatrix@=0.3pc{
{\bullet} \ar@/_1pc/[dddd] \ar@/^1pc/[dddd] \\ \\ \\ \\
{\bullet}
}
\end{array}
&
\begin{array}{c}
\xymatrix@=0.3pc{
{\bullet} \ar@/_1pc/[dddd] \ar[drr]
&& && {\bullet} \ar[llll] \ar@/^1pc/[dddd] \\
&& {\bullet} \ar[dd] \ar[urr] \\ \\
&& {\bullet} \ar[dll] \\
{\bullet} \ar[rrrr] && && {\bullet} \ar[ull]
}
\end{array}
&
\begin{array}{c}
\xymatrix@=0.3pc{
{\bullet} \ar@/_1pc/[dddd] \ar[drr]
&& && {\bullet} \ar[drr] \ar[llll]
&& && {\bullet} \ar[llll] \ar@/^1pc/[dddd] \\
&& {\bullet} \ar[dd] \ar[urr]
&& && {\bullet} \ar[dd] \ar[urr] \\ \\
&& {\bullet} \ar[dll]
&& && {\bullet} \ar[dll] \\
{\bullet} \ar[rrrr]
&& && {\bullet} \ar[ull] \ar[rrrr]
&& && {\bullet} \ar[ull]
}
\end{array}
\\
(0,2) & (0,3) & (0,4)
\\ \\
\begin{array}{c}
\xymatrix@=0.3pc{
&&& {\bullet} \ar[ddl] \ar@/^1pc/[dddd] \\ \\
{\bullet} \ar[uurrr] \ar[ddrrr]
&& {\bullet} \ar@/_0.1pc/[ll] \ar@/^0.1pc/[ll] \\ \\
&&& {\bullet} \ar[uul]
}
\end{array}
&
\begin{array}{c}
\xymatrix@=0.3pc{
&&& {\bullet} \ar[ddl] \ar[drr]
&& && {\bullet} \ar[llll] \ar@/^1pc/[dddd] \\
&&& && {\bullet} \ar[dd] \ar[urr] \\
{\bullet} \ar[uurrr] \ar[ddrrr]
&& {\bullet} \ar@/_0.1pc/[ll] \ar@/^0.1pc/[ll] \\
&&& && {\bullet} \ar[dll] \\
&&& {\bullet} \ar[uul] \ar[rrrr]
&& && {\bullet} \ar[ull]
}
\end{array}
&
\begin{array}{c}
\xymatrix@=0.3pc{
&&& {\bullet} \ar[ddl] \ar[drr]
&& && {\bullet} \ar[llll] \ar[drr]
&& && {\bullet} \ar[llll] \ar@/^1pc/[dddd] \\
&&& && {\bullet} \ar[dd] \ar[urr]
&& && {\bullet} \ar[dd] \ar[urr] \\
{\bullet} \ar[uurrr] \ar[ddrrr]
&& {\bullet} \ar@/_0.1pc/[ll] \ar@/^0.1pc/[ll] \\
&&& && {\bullet} \ar[dll] && && {\bullet} \ar[dll] \\
&&& {\bullet} \ar[uul] \ar[rrrr]
&& && {\bullet} \ar[ull] \ar[rrrr]
&& && {\bullet} \ar[ull]
}
\end{array}
\\
(1,1) & (1,2) & (1,3)
\\ \\
\begin{array}{c}
\xymatrix@=0.79pc{
&& {\bullet} \ar[dl] \ar[dd] \\
{\bullet} \ar[urr] \ar[drr]
& {\bullet} \ar@/_0.1pc/[l] \ar@/^0.1pc/[l]
& & {\bullet} \ar[ul] \ar@/^1pc/[ddd] \\
&& {\bullet} \ar[ur] \ar[ul] \\
&& {\bullet} \ar[dl] \ar[dd] \\
{\bullet} \ar[urr] \ar[drr]
& {\bullet} \ar@/_0.1pc/[l] \ar@/^0.1pc/[l]
& & {\bullet} \ar[ul] \\
&& {\bullet} \ar[ur] \ar[ul]
}
\end{array}
&
\begin{array}{c}
\xymatrix@=0.79pc{
&& {\bullet} \ar[dl] \ar[dd] \\
{\bullet} \ar[urr] \ar[drr]
& {\bullet} \ar@/_0.1pc/[l] \ar@/^0.1pc/[l]
& & {\bullet} \ar[ul] \ar[dr]
&& {\bullet} \ar[ll] \ar@/^1pc/[ddd] \\
&& {\bullet} \ar[ur] \ar[ul] && {\bullet} \ar[ur] \ar[d] \\
&& {\bullet} \ar[dl] \ar[dd] && {\bullet} \ar[dl] \\
{\bullet} \ar[urr] \ar[drr]
& {\bullet} \ar@/_0.1pc/[l] \ar@/^0.1pc/[l]
& & {\bullet} \ar[ul] \ar[rr] && {\bullet} \ar[ul] \\
&& {\bullet} \ar[ur] \ar[ul]
}
\end{array}
&
\begin{array}{c}
\xymatrix@=0.79pc{
&& {\bullet} \ar[dl] \ar[dd] \\
{\bullet} \ar[urr] \ar[drr]
& {\bullet} \ar@/_0.1pc/[l] \ar@/^0.1pc/[l]
& & {\bullet} \ar[ul] \ar[dr]
&& {\bullet} \ar[ll] \ar[dr] && {\bullet} \ar[ll] \ar@/^1pc/[ddd] \\
&& {\bullet} \ar[ur] \ar[ul] && {\bullet} \ar[ur] \ar[d]
&& {\bullet} \ar[ur] \ar[d] \\
&& {\bullet} \ar[dl] \ar[dd] && {\bullet} \ar[dl]
&& {\bullet} \ar[dl] \\
{\bullet} \ar[urr] \ar[drr]
& {\bullet} \ar@/_0.1pc/[l] \ar@/^0.1pc/[l]
& & {\bullet} \ar[ul] \ar[rr] && {\bullet} \ar[ul] \ar[rr]
&& {\bullet} \ar[ul] \\
&& {\bullet} \ar[ur] \ar[ul]
}
\end{array}
\\
(2,1) & (2,2) & (2,3)
\end{array}
\]
\caption{Representative quivers in each $\cQ_{g,b}$ for small values of
$(g,b)$, $b>0$.}
\label{fig:Qgb}
\end{figure}
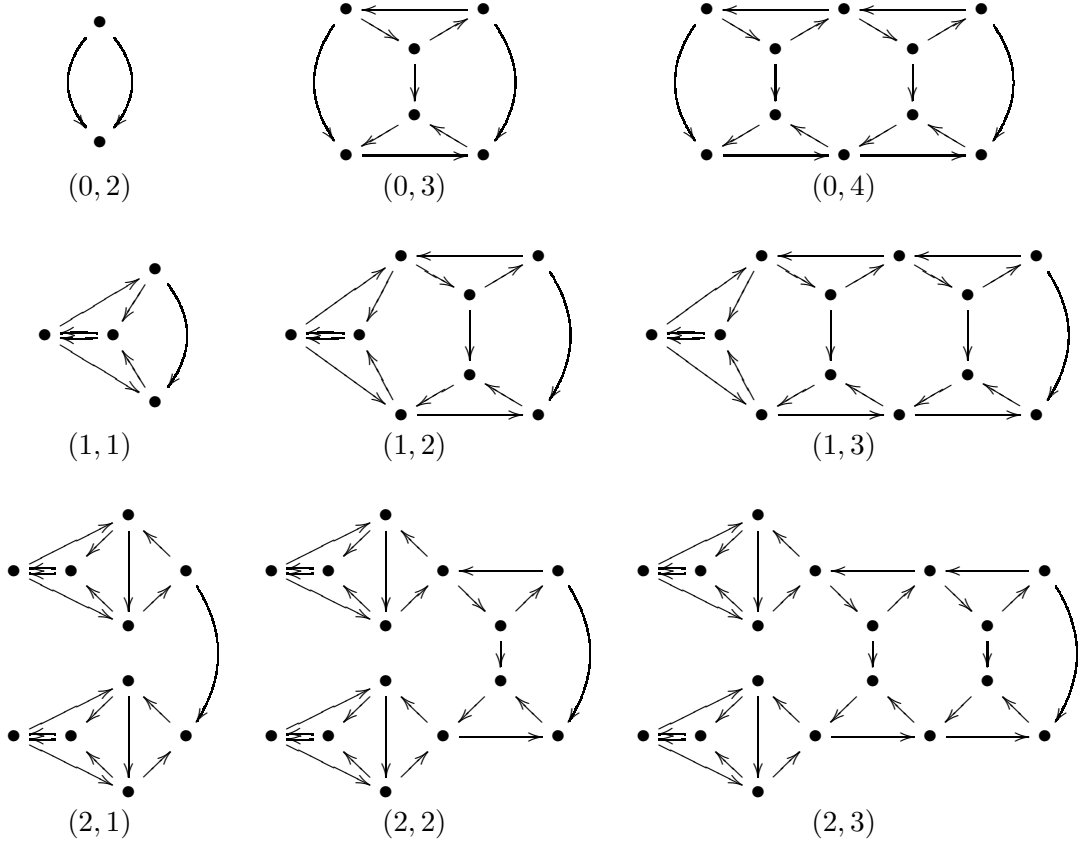

The next theorem provides a complete answer to our question.
\begin{theorem} \label{t:const}
For a connected quiver $Q$, the following statements are equivalent.
\begin{enumerate}
\renewcommand{\theenumi}{\roman{enumi}}
\item \label{it:mut:same}
All the quivers in the mutation class of $Q$ have the same number of
arrows.

\item \label{it:mut:Qgb}
$Q$ has at most two vertices, or $Q \in \cQ_{g,b}$ for some $g,b \geq 0$
such that $(g,b) \not \in \{(0,0), (0,1)\}$.
\end{enumerate}

Moreover, any quiver in $\cQ_{g,b}$ for $b>0$ has $6(g-1)+4b$ vertices
and $12(g-1)+7b$ arrows, whereas any quiver in $\cQ_{g,0}$ has $6g-3$
vertices, $12g-6$ arrows and any of its vertices has exactly two
incoming and two outgoing arrows.
\end{theorem}

From the theorem we see that mutation classes consisting of connected
quivers with constant number of arrows are quite rare. In fact, for any
$n \geq 3$ the number of such classes with $n$ vertices is finite. We
also deduce the following.

\begin{cor*}
Let $n>1$. There exists a mutation class consisting of connected
quivers with $n$ vertices and constant number of arrows if and only if
$n \not \equiv 1,5 \pmod 6$.
\end{cor*}

\subsection{Outline of the proof}

A mutation class whose quivers have the same number of arrows must
consist of a finite number of quivers. The proof of the implication
(\ref{it:mut:same}) $\Rightarrow$ (\ref{it:mut:Qgb}) in
Theorem~\ref{t:const} relies on the classification by Felikson, Shapiro
and Tumarkin~\cite{FST08b} of the connected quivers whose mutation
classes are finite; either they arise from triangulations of marked
bordered oriented surfaces, or they are mutation equivalent to one of
$11$ exceptional quivers, or they are acyclic with two vertices and at
least three arrows between them.

Obviously, mutations of quivers with one or two vertices are just
reflections preserving the number of arrows. Next, one checks that each
of the $11$ exceptional mutation classes contains two quivers with
different numbers of arrows, see Section~\ref{sec:except}.

We are left to consider quivers arising from triangulations of marked
bordered surfaces. We show in Section~\ref{sec:addpoint} that for such
marked surface $(S,M')$ admitting a triangulation, adding a puncture
or, under mild condition, a marked point on its boundary, results in a
marked surface $(S,M)$ which has two triangulations giving rise to two
quivers in the corresponding mutation class with different numbers of
arrows. Therefore we need only to consider ``minimal'' marked bordered
surfaces, and those that remain are precisely the ones giving rise to
the mutation classes $\cQ_{g,b}$ defined above.

To prove the implication (\ref{it:mut:Qgb}) $\Rightarrow$
(\ref{it:mut:same}) in Theorem~\ref{t:const} for these classes, we show
in Section~\ref{sec:Qgb} that mutating any of their quivers at any
vertex does not change the number of arrows. Note that for the classes
$\cQ_{g,b}$ with $b>0$ the statement of the theorem has already been
shown in~\cite[\S 2]{Ladkani11} by a counting argument yielding
formulae for the numbers of arrows and $3$-cycles in any quiver in
$\cQ_{g,b}$ which depend only on $g$ and $b$.

In this paper we take a more general approach and characterize, for the
family of quivers arising from triangulations of marked bordered
surfaces without punctures, those mutations that preserve the number of
arrows. This is done by analyzing the possible local ``neighborhoods''
at any vertex. The result for $\cQ_{g,b}$ with $b>0$ then follows as a
special case, see Section~\ref{sec:unpunctured}.

Similar analysis for the remaining classes $\cQ_{g,0}$ is done
in Section~\ref{sec:Qg0}, where we also explore their relations with
mutation classes of quivers of Dynkin type $A$.

\subsection{Algebraic interpretation -- derived equivalence}

By using the theory of quivers with potentials (QP) and their mutations
developed by Derksen, Weyman and Zelevinsky~\cite{DWZ08} it is
possible, in certain cases, to interpret mutations of QP preserving the
number of arrows as derived equivalences of the corresponding Jacobian
algebras. In particular, to any quiver in the mutation classes
$\cQ_{g,b}$ with $(g,b) \neq (0,0),(0,1)$ there is a naturally
associated potential allowing each class $\cQ_{g,b}$ to be regarded as
a mutation classes of QP whose Jacobian algebras are all derived
equivalent. This is elaborated in Section~\ref{sec:dereq}.

\section{Mutation classes with varying number of arrows}
\label{sec:varying}

In this section we prove the implication (\ref{it:mut:same})
$\Rightarrow$ (\ref{it:mut:Qgb}) in Theorem~\ref{t:const}. We start
with the following useful lemma.

\begin{lemma} \label{l:in1out1}
Let $Q$ be a quiver and let $k$ be a vertex in $Q$ with exactly one
incoming arrow and exactly one outgoing arrow. Then the numbers of
arrows in $Q$ and in $\mu_k(Q)$ differ by one.
\end{lemma}
\begin{proof}
Denote by $i \xrightarrow{\alpha} k$ the incoming arrow and by $k
\xrightarrow{\beta} j$ the outgoing arrow. Apart from inverting the
arrows $\alpha$ and $\beta$, the quiver mutation at $k$ modifies only
the arrows between $i$ and $j$.

If $Q$ has $a \geq 0$ arrows from $i$ to $j$, its mutation $\mu_k(Q)$
has $a+1$ arrows from $i$ to $j$, whereas if $Q$ has $a \geq 1$ arrows
from $j$ to $i$, its mutation $\mu_k(Q)$ has $a-1$ arrows from $j$ to
$i$. In any case, the number of arrows changes by one.
\end{proof}

\subsection{The exceptional quivers}
\label{sec:except}

\begin{lemma}
The mutation class of each of the $11$ exceptional quivers contains two
quivers with different numbers of arrows.
\end{lemma}
\begin{proof}
In principle, since each mutation class is finite, the claim can be
verified on a computer. For the convenience of the reader, we give a
direct proof.

Since each of the $8$ quivers $E_6$, $E_7$, $E_8$, $\tilde{E}_6$,
$\tilde{E}_7$, $\tilde{E}_8$, $E^{(1,1)}_7$ and $E^{(1,1)}_8$ contains
a vertex with exactly one incoming arrow and exactly one outgoing
arrow, by Lemma~\ref{l:in1out1} we can mutate at this vertex and get a
quiver with a different number of arrows.

The quiver $E^{(1,1)}_6$ is shown in the left picture below. Note that
performing any single mutation on this quiver does not change the
number of arrows. However, when mutating at the vertex $1$ and then at
$2$, we get a quiver with one arrow less, as shown in the right
picture.
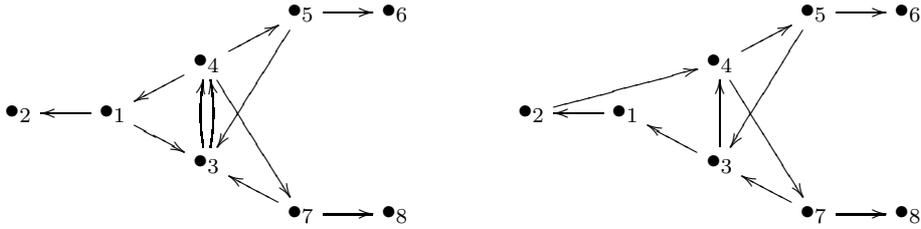
\begin{align*}
\xymatrix@=0.5pc{
&& && && {\bullet_5}
\ar[dddll] \ar[rr] && {\bullet_6} \\
&& && {\bullet_4} \ar[urr] \ar[dddrr] \ar[dll] \\
{\bullet_2} && {\bullet_1} \ar[ll] \ar[drr] \\
&& && {\bullet_3} \ar@/^0.2pc/[uu] \ar@/_0.2pc/[uu] \\
&& && && {\bullet_7} \ar[ull] \ar[rr] && {\bullet_8}
}
& &
\xymatrix@=0.5pc{
&& && && {\bullet_5}
\ar[dddll] \ar[rr] && {\bullet_6} \\
&& && {\bullet_4} \ar[urr] \ar[dddrr] \\
{\bullet_2} \ar[urrrr] && {\bullet_1} \ar[ll] \\
&& && {\bullet_3} \ar[uu] \ar[ull] \\
&& && && {\bullet_7} \ar[ull] \ar[rr] && {\bullet_8}
}
\end{align*}

It remains to consider the quivers $X_6$ and $X_7$ introduced by
Derksen and Owen~\cite{DerksenOwen08} who also wrote down explicitly
their mutation classes which consist of $5$ and $2$ quivers,
respectively. We see that the mutation class of $X_6$ contains quivers
with $9$ and $11$ arrows, whereas that of $X_7$ contains ones with $12$
and $15$ arrows.
\end{proof}

\subsection{Quivers from marked surfaces}
\label{sec:addpoint}

Consider now quivers arising from triangulations of bordered oriented
surfaces with marked points. We refer the reader to~\cite{FST08} for
the relevant definitions and constructions.


\begin{lemma} \label{l:addp}
Let $(S,M')$ be a marked surface which has a triangulation, and let $M$
be obtained from $M'$ by adding a puncture. Then $(S,M)$ has two
triangulations whose corresponding quivers have different numbers of
arrows.
\end{lemma}
\begin{proof}
Choose an arc $\gamma$ of a triangulation $T'$ of $(S,M')$ and denote
the marked points at its ends by $A$ and $B$ (which may coincide). We
may place the new puncture $\times$ in $M$ on $\gamma$ and obtain a
triangulation $T$ of $(S,M)$ by replacing $\gamma$ with four arcs as
depicted in the left picture below. By flipping the arc labeled $2$, we
obtain another triangulation of $(S,M)$, part of which is depicted in
the right picture.
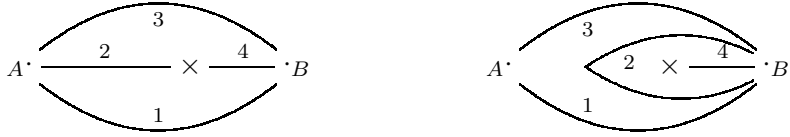
\begin{align*}
\xymatrix{
{_A\cdot} \ar@{-}[rr]^2 \ar@{-}@/^2pc/[rrr]_3 \ar@{-}@/_2pc/[rrr]^1
& *=0{} & {\times} \ar@{-}[r]^4
& {\cdot_B}
}
& &
\xymatrix{
{_A\cdot} \ar@{-}@/^2pc/[rrr]_(0.3){3} \ar@{-}@/_2pc/[rrr]^(0.3){1}
& *=0{} & {\times} \ar@{-}[r]^4
& {\cdot_B} \ar@{-}@/_1pc/[ll]^(0.8){2} \ar@{-}@/^1pc/[ll]
}
\end{align*}

Consider the quiver $Q$ corresponding to $T$ and its mutation
$\mu_2(Q)$ corresponding to the flip of $T$ at the arc labeled $2$.
Since $Q$ has precisely one arrow ending at $2$, namely the arrow $1
\to 2$ induced from the triangle $\{1,2,4\}$ in $T$, and precisely one
arrow starting at $2$, namely $2 \to 3$ induced from the triangle
$\{2,3,4\}$, by Lemma~\ref{l:in1out1} the numbers of arrows in $Q$ and
in $\mu_2(Q)$ differ by one.
\end{proof}


\begin{lemma} \label{l:addb}
Let $(S,M')$ be a marked surface with non-empty boundary which has a
triangulation satisfying the following condition:
\begin{enumerate}
\item[($\spadesuit$)]
There is a triangle which is not self-folded such that exactly one of
its sides is a boundary segment.
\end{enumerate}

Denoting this segment by $c$, let $M$ be obtained from $M'$ by adding a
marked point on $c$. Then $(S,M)$ has two triangulations whose
corresponding quivers have different numbers of arrows. Moreover, one
of these triangulations satisfies the condition~($\spadesuit$).
\end{lemma}
\begin{proof}
Let $T'$ be a triangulation of $(S,M')$ satisfying
condition~($\spadesuit$). The triangle of $T'$ from that condition
looks like
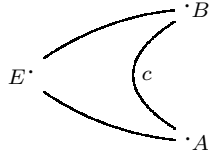
\[
\xymatrix@=1pc{
& & & {\cdot_B} \ar@{-}@/_2pc/[dd]^{c} \\
{_E\cdot} \ar@{-}@/_/[drrr] \ar@{-}@/^/[urrr] \\
& & & {\cdot_A}
}
\]
where $A$ and $B$ denote the endpoints of the boundary segment $c$ and
$E$ is the other vertex of the triangle (note that all these marked
points might coincide).

Let $M = M' \cup \{Z\}$ where $Z$ is a point on the segment $c$,
splitting it into two segments $c'$ and $c''$. We obtain from $T'$ a
triangulation $T$ of $(S,M)$ by adding the arc connecting $E$ and $Z$
as in the left picture below. The triangle consisting of the arcs
labeled $1$, $2$ and the segment $c'$ shows that $T$ satisfies
condition~($\spadesuit$). By flipping the arc labeled $2$, we obtain
another triangulation of $(S,M)$, part of which is depicted in the
right picture.
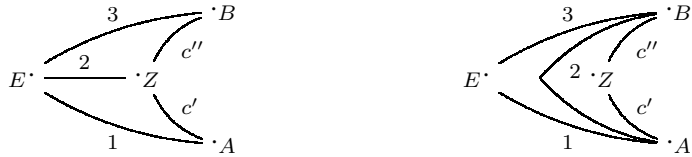
\begin{align*}
\xymatrix@=1pc{
& & & {\cdot_B} \ar@{-}@/_/[dl]^{c''} \\
{_E\cdot} \ar@{-}@/_/[drrr]_1 \ar@{-}@/^/[urrr]^3 \ar@{-}[rr]^2 & &
{\cdot_Z} \ar@{-}@/_/[dr]^{c'} \\
& & & {\cdot_A}
}
& &
\xymatrix@=1pc{
& & & {\cdot_B} \ar@{-}@/_/[dl]^{c''} \\
{_E\cdot} \ar@{-}@/_/[drrr]_1 \ar@{-}@/^/[urrr]^3 &
*=0{} \ar@{-}@/^/[urr]_(0.2){2} \ar@{-}@/_/[drr] &
{\cdot_Z} \ar@{-}@/_/[dr]^{c'} \\
& & & {\cdot_A}
}
\end{align*}

Consider the quiver $Q$ corresponding to $T$ and its mutation
$\mu_2(Q)$ corresponding to the flip of $T$ at the arc labeled $2$.
Since $Q$ has precisely one arrow ending at $2$, namely the arrow $1
\to 2$ induced from the triangle $\{1,2,c'\}$ in $T$, and precisely one
arrow starting at $2$, namely $2 \to 3$ induced from the triangle
$\{2,3,c''\}$, by Lemma~\ref{l:in1out1} the numbers of arrows in $Q$
and in $\mu_2(Q)$ differ by one.
\end{proof}

\begin{lemma}
Let $(S,M)$ be a marked surface without boundary whose triangulations
give rise to a mutation class consisting of quivers with the same
number of arrows. Then $(S,M)$ must be one of the following:
\begin{enumerate}
\renewcommand{\theenumi}{\alph{enumi}}
\item
A sphere with $3$ punctures. The corresponding quiver is then a
disjoint union of three vertices;

\item
A (closed) surface of genus $g>0$ with one puncture. The corresponding
mutation class is $\cQ_{g,0}$.
\end{enumerate}
\end{lemma}
\begin{proof}
It follows from Lemma~\ref{l:addp} that for any puncture $P$ of $M$,
the marked surface $(S, M \setminus \{P\})$ does not have any
triangulation. Hence $M$ must be minimal with respect to the property
that $(S,M)$ still has a triangulation.
\end{proof}

\begin{lemma}
Let $(S,M)$ be a marked surface with non-empty boundary whose
triangulations give rise to a mutation class consisting of quivers with
the same number of arrows. Then $(S,M)$ must be one of the following:
\begin{enumerate}
\renewcommand{\theenumi}{\alph{enumi}}
\item
A disc with one puncture and at most two marked points on its boundary.
The corresponding quivers are $A_1$ and $A_2$;

\item
A disc with no punctures and four or five marked points on its
boundary. The corresponding quivers are $A_1$ and $A_2$;

\item
A surface of genus $g$ with $b>0$ boundary components such that $(g,b)
\neq (0,1)$ with exactly one marked point on each boundary component
and no punctures. The corresponding mutation class is $\cQ_{g,b}$.
\end{enumerate}
\end{lemma}
\begin{proof}
Applying Lemma~\ref{l:addp} we deduce that the number $b$ of boundary
components of $S$, its genus $g$ and the number of punctures in $M$
must be one of the following:
\begin{itemize}
\item
$b=1$, $g=0$ (disc) with one puncture;

\item
$b=1$, $g=0$ (disc) with no punctures;

\item
$g=0$, $b \geq 2$ with no punctures;

\item
$g>0$, $b \geq 1$ with no punctures.
\end{itemize}

The following bordered marked surfaces of genus $g$ with $b>0$ boundary
components admit triangulations satisfying the
condition~($\spadesuit$):
\begin{itemize}
\item
$b=1$, $g=0$ (disc) with one puncture and two marked points on its
boundary;

\item
$b=1$, $g=0$ (disc) with no punctures and five marked points on its
boundary;

\item
$(g,b) \neq (0,1)$ with one marked point on each boundary component.
\end{itemize}
By applying Lemma~\ref{l:addb} we see that the numbers of marked points
on each boundary component of $S$ cannot exceed the ones in the list
above, and the result follows.
\end{proof}

\section{On the mutation classes $\cQ_{g,b}$}
\label{sec:Qgb}

In order to show the implication (\ref{it:mut:Qgb}) $\Rightarrow$
(\ref{it:mut:same}) in Theorem~\ref{t:const}, it remains to show that
each of the mutation classes $\cQ_{g,b}$ for $(g,b) \not \in \{(0,0),
(0,1)\}$ consists of quivers with the same number of arrows.

\subsection{Mutations of quivers from unpunctured marked surfaces}
\label{sec:unpunctured}

For the classes $\cQ_{g,b}$ with $b>0$ this has already been shown in
our previous work~\cite{Ladkani11}, but we give here a different proof
by characterizing the mutations of quivers arising from triangulations
of marked bordered surfaces without punctures that preserve the number
of arrows. In fact, we show that for such quivers the converse of
Lemma~\ref{l:in1out1} holds.

For a vertex $k$ in a quiver $Q$, recall that its \emph{in-degree} is
the number of arrows ending at $k$. Similarly, its \emph{out-degree} is
the number of arrows starting at $k$. The \emph{neighborhood} of $k$ is
the full subquiver of $Q$ on the set of vertices consisting of $k$, the
vertices $i$ having arrows $i \to k$ and the vertices $j$ having arrows
$k \to j$. Mutation at $k$ does not change any arrows outside the
neighborhood of $k$. Thus, when assessing its effect on the number of
arrows, it is enough to consider its effect on the neighborhood of $k$.

\begin{prop} \label{p:in1out1}
Let $Q$ be a quiver arising from a triangulation of a marked bordered
surface without punctures. For a vertex $k$ of $Q$, the following
conditions are equivalent:
\begin{enumerate}
\renewcommand{\theenumi}{\roman{enumi}}
\item
$k$ has in-degree $1$ and out-degree $1$.

\item
The quivers $Q$ and $\mu_k(Q)$ do not have the same number of arrows.
\end{enumerate}
\end{prop}

\begin{table}
\[
\begin{array}{|c|c|c||c|c|}
\hline
%
%
1 &
\begin{array}{c}
\xymatrix@=3pc{
{\cdot} \ar@{-}[r]^{i} \ar@<-0.5ex>@{.}[d] \ar@{-}[d]
& {\cdot} \ar@<0.5ex>@{.}[d] \ar@{-}[d] \\
{\cdot} \ar@{-}[r] \ar@<-0.5ex>@{.}[r] \ar@{-}[ur]^k & {\cdot}
}
\end{array}
&
\begin{array}{c}
\xymatrix{
{i} \ar[r] & {k}
}
\end{array}
&
\begin{array}{c}
\xymatrix{
{i} & {k} \ar[l]
}
\end{array}
&
\begin{array}{c}
\xymatrix@=3pc{
{\cdot} \ar@{-}[r]^{i} \ar@<-0.5ex>@{.}[d] \ar@{-}[d]
& {\cdot} \ar@<0.5ex>@{.}[d] \ar@{-}[d] \\
{\cdot} \ar@{-}[r] \ar@<-0.5ex>@{.}[r] & {\cdot} \ar@{-}[ul]_k
}
\end{array}
\\ \hline
%
%
2a &
\begin{array}{c}
\xymatrix@=3pc{
{\cdot} \ar@{-}[r]^{i} \ar@<-0.5ex>@{.}[d] \ar@{-}[d]
& {\cdot} \ar@<0.5ex>@{.}[d] \ar@{-}[d] \\
{\cdot} \ar@{-}[r]_{i} \ar@{-}[ur]^k & {\cdot}
}
\end{array}
&
\begin{array}{c}
\xymatrix{
{i} \ar@<-0.5ex>[r] \ar@<0.5ex>[r] & {k}
}
\end{array}
&
\begin{array}{c}
\xymatrix{
{i} & {k} \ar@<-0.5ex>[l] \ar@<0.5ex>[l]
}
\end{array}
&
\begin{array}{c}
\xymatrix@=3pc{
{\cdot} \ar@{-}[r]^{i} \ar@<-0.5ex>@{.}[d] \ar@{-}[d]
& {\cdot} \ar@<0.5ex>@{.}[d] \ar@{-}[d] \\
{\cdot} \ar@{-}[r]_{i} & {\cdot} \ar@{-}[ul]_k
}
\end{array}
\\ \hline
2b &
\begin{array}{c}
\xymatrix@=3pc{
{\cdot} \ar@{-}[r]^{i_1} \ar@<-0.5ex>@{.}[d] \ar@{-}[d]
& {\cdot} \ar@<0.5ex>@{.}[d] \ar@{-}[d] \\
{\cdot} \ar@{-}[r]_{i_2} \ar@{-}[ur]^k & {\cdot}
}
\end{array}
&
\begin{array}{c}
\xymatrix@R=1pc{
{i_1} \ar[dr] \\
& {k} \\
{i_2} \ar[ur]
}
\end{array}
&
\begin{array}{c}
\xymatrix@R=1pc{
{i_1} \\
& {k} \ar[ul] \ar[dl] \\
{i_2}
}
\end{array}
&
\begin{array}{c}
\xymatrix@=3pc{
{\cdot} \ar@{-}[r]^{i_1} \ar@<-0.5ex>@{.}[d] \ar@{-}[d]
& {\cdot} \ar@<0.5ex>@{.}[d] \ar@{-}[d] \\
{\cdot} \ar@{-}[r]_{i_2} & {\cdot} \ar@{-}[ul]_k
}
\end{array}
\\ \hline
2c &
\begin{array}{c}
\xymatrix@=3pc{
{\cdot} \ar@{-}[r]^{i} \ar@{-}[d]_{j}
& {\cdot} \ar@<0.5ex>@{.}[d] \ar@{-}[d] \\
{\cdot} \ar@<-0.5ex>@{.}[r] \ar@{-}[r] \ar@{-}[ur]^k & {\cdot}
}
\end{array}
&
\begin{array}{c}
\xymatrix@R=1pc{
{i} \ar[dr] \\
& {k} \ar[dl] \\
{j} \ar[uu]
}
\\
a_{ji} \geq 1
\end{array}
&
\begin{array}{c}
\xymatrix@R=1pc{
{i} \\
& {k} \ar[ul] \\
{j} \ar[ur]
}
\\
a_{ij} = 0
\end{array}
&
\begin{array}{c}
\xymatrix@=3pc{
{\cdot} \ar@{-}[r]^{i} \ar@{-}[d]_{j}
& {\cdot} \ar@<0.5ex>@{.}[d] \ar@{-}[d] \\
{\cdot} \ar@<-0.5ex>@{.}[r] \ar@{-}[r] & {\cdot} \ar@{-}[ul]_k
}
\end{array}
\\ \hline
%
%
3a &
\begin{array}{c}
\xymatrix@=3pc{
{\cdot} \ar@{-}[r]^{i} \ar@{-}[d]_{j}
& {\cdot} \ar@<0.5ex>@{.}[d] \ar@{-}[d] \\
{\cdot} \ar@{-}[r]_{i} \ar@{-}[ur]^k & {\cdot}
}
\end{array}
&
\begin{array}{c}
\xymatrix@R=1pc{
{i} \ar@<-0.5ex>[dr] \ar@<0.5ex>[dr] \\
& {k} \ar[dl] \\
{j} \ar[uu]
}
\\
a_{ji} = 1
\end{array}
&
\begin{array}{c}
\xymatrix@R=1pc{
{i} \ar[dd] \\
& {k} \ar@<-0.5ex>[ul] \ar@<0.5ex>[ul] \\
{j} \ar[ur]
}
\\
a_{ij} = 1
\end{array}
&
\begin{array}{c}
\xymatrix@=3pc{
{\cdot} \ar@{-}[r]^{i} \ar@{-}[d]_{j}
& {\cdot} \ar@<0.5ex>@{.}[d] \ar@{-}[d] \\
{\cdot} \ar@{-}[r]_{i} & {\cdot} \ar@{-}[ul]_k
}
\end{array}
\\ \hline
3b &
\begin{array}{c}
\xymatrix@=3pc{
{\cdot} \ar@{-}[r]^{i_1} \ar@{-}[d]_{j}
& {\cdot} \ar@<0.5ex>@{.}[d] \ar@{-}[d] \\
{\cdot} \ar@{-}[r]_{i_2} \ar@{-}[ur]^k & {\cdot}
}
\end{array}
&
\begin{array}{c}
\xymatrix@R=1pc{
{i_1} \ar[dr] \\
& {k} \ar[dl]  \\
{j} \ar[uu] & & {i_2} \ar[ul]
}
\
\\
a_{ji_1} \geq 1,\, a_{ji_2}=0
\end{array}
&
\begin{array}{c}
\xymatrix@R=1pc{
{i_1} \\
& {k} \ar[ul] \ar[dr] \\
{j} \ar[ur] & & {i_2} \ar[ll]
}
\\
a_{i_1 j} = 0,\, a_{i_2 j} \geq 1
\end{array}
&
\begin{array}{c}
\xymatrix@=3pc{
{\cdot} \ar@{-}[r]^{i_1} \ar@{-}[d]_{j}
& {\cdot} \ar@<0.5ex>@{.}[d] \ar@{-}[d] \\
{\cdot} \ar@{-}[r]_{i_2} & {\cdot} \ar@{-}[ul]_k
}
\end{array}
\\ \hline
\end{array}
\]
\caption{Neighborhoods of $k$ when at least one side is a boundary
segment.} \label{t:segments}
\end{table}

\begin{proof}
In the triangulation corresponding to $Q$, the arc corresponding to the
vertex $k$ is a side of two triangles whose other sides are denoted
$i'$, $j'$ and $i''$, $j''$ as shown below:
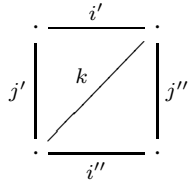
\[
\xymatrix@=3pc{
{\cdot} \ar@{-}[r]^{i'} \ar@{-}[d]_{j'} & {\cdot} \ar@{-}[d]^{j''} \\
{\cdot} \ar@{-}[r]_{i''} \ar@{-}[ur]^k & {\cdot} }
\]

Some of these sides may be boundary segments. In addition, it may
happen that the sides denoted $i'$ and $i''$ are in fact the same arc
in the triangulation, and similarly for $j'$ and $j''$.

The proof goes by examining all the possible cases and computing the
corresponding neighborhoods of $k$ and their mutations. The details of
this verification are given in a concise form in Table~\ref{t:segments}
for the cases where at least one of the sides is a boundary segment and
in Table~\ref{t:nosegments} for the remaining cases where all the sides
are arcs.

Each row in these tables represents one such case and its mutation. The
left two columns show the triangles whose side is the arc $k$ together
with (part of) the corresponding neighborhood of $k$. The right two
columns show what happens under mutation at $k$, which corresponds to a
flip of the arc $k$. Note that mutation of each of the cases~(4a) and (4c)
leads to the same case (up to relabeling of vertices/arcs), hence
there are actually only four different cases (and not six) in
Table~\ref{t:nosegments}.

In the pictures, we use the following conventions. When drawing the
triangulation, a line of the form $\xymatrix@1{\ar@{-}[r] &}$ denotes
an arc and a line of the form $\xymatrix@1{\ar@<0.5ex>@{.}[r]
\ar@{-}[r] &}$ denotes a boundary segment. As for the corresponding
quiver, we draw only the arrows that must be present in the
neighborhood of $k$, in particular all the arrows starting or
terminating at $k$, and complement this by writing down some extra
conditions that must be also satisfied. These are given in terms of
numbers of arrows between vertices, where we denote by $a_{ij}$ the
number of arrows starting at the vertex $i$ and ending at $j$.

Most of these constraints follow from the fact that the marked points
lie on the boundary and are not punctures. In addition, in case~(3a) it
cannot happen that $a_{ji}=2$ since otherwise the neighborhood of $i$
would fall into case~(4a), but there is only one arrow $k \to j$. In
case~(4b) we have $a_{j i_1} \geq 1$ and $a_{j i_2} \geq 1$, but these
are actually equalities since otherwise the vertex $j$ would have too
many outgoing arrows. Finally, note that case~(4a) actually never
occurs for an unpunctured surface, as it leads to a triangulation of
the torus with one puncture.
\end{proof}

\begin{table}
\[
\begin{array}{|c|c|c||c|c|}
\hline
4a &
\begin{array}{c}
\xymatrix@=3pc{
{\cdot} \ar@{-}[r]^{i} \ar@{-}[d]_{j}
& {\cdot} \ar@{-}[d]_{j} \\
{\cdot} \ar@{-}[r]_{i} \ar@{-}[ur]^k & {\cdot}
}
\end{array}
&
\begin{array}{c}
\xymatrix@R=1pc{
{i} \ar@<-0.5ex>[dr] \ar@<0.5ex>[dr] \\
& {k} \ar@<-0.5ex>[dl] \ar@<0.5ex>[dl] \\
{j} \ar@<-0.5ex>[uu] \ar@<0.5ex>[uu]
}
\end{array}
&
\begin{array}{c}
\xymatrix@R=1pc{
{i} \ar@<-0.5ex>[dd] \ar@<0.5ex>[dd] \\
& {k} \ar@<-0.5ex>[ul] \ar@<0.5ex>[ul] \\
{j} \ar@<-0.5ex>[ur] \ar@<0.5ex>[ur]
}
\end{array}
&
\begin{array}{c}
\xymatrix@=3pc{
{\cdot} \ar@{-}[r]^{i} \ar@{-}[d]^{j}
& {\cdot} \ar@{-}[d]^{j} \\
{\cdot} \ar@{-}[r]_{i} & {\cdot} \ar@{-}[ul]_k
}
\end{array}
\\ \hline
4b &
\begin{array}{c}
\xymatrix@=3pc{
{\cdot} \ar@{-}[r]^{i_1} \ar@{-}[d]_{j}
& {\cdot} \ar@{-}[d]_{j} \\
{\cdot} \ar@{-}[r]_{i_2} \ar@{-}[ur]^k & {\cdot}
}
\end{array}
&
\begin{array}{c}
\xymatrix@R=1pc{
{i_1} \ar[dr] \\
& {k} \ar@<-0.5ex>[dl] \ar@<0.5ex>[dl] \\
{j} \ar[uu] \ar[rr] & & {i_2} \ar[ul]
}
\\
a_{j i_1} = a_{j i_2} = 1
\end{array}
&
\begin{array}{c}
\xymatrix@R=1pc{
{i_1} \ar[dd] \\
& {k} \ar[ul] \ar[dr] \\
{j} \ar@<-0.5ex>[ur] \ar@<0.5ex>[ur] & & {i_2} \ar[ll]
}
\\
a_{i_1 j} = a_{i_2 j} = 1
\end{array}
&
\begin{array}{c}
\xymatrix@=3pc{
{\cdot} \ar@{-}[r]^{i_1} \ar@{-}[d]^{j}
& {\cdot} \ar@{-}[d]^{j} \\
{\cdot} \ar@{-}[r]_{i_2} & {\cdot} \ar@{-}[ul]_k
}
\end{array}
\\ \hline
4c &
\begin{array}{c}
\xymatrix@=3pc{
{\cdot} \ar@{-}[r]^{i_1} \ar@{-}[d]_{j_1}
& {\cdot} \ar@{-}[d]_{j_2} \\
{\cdot} \ar@{-}[r]_{i_2} \ar@{-}[ur]^k & {\cdot}
}
\end{array}
&
\begin{array}{c}
\xymatrix@R=1pc{
{i_1} \ar[dr] & & {j_2} \ar[dd] \\
& {k} \ar[dl] \ar[ur] \\
{j_1} \ar[uu] & & {i_2} \ar[ul]
}
\\
a_{j_1 i_1} \geq 1,\, a_{j_2 i_2} \geq 1
\\
a_{j_1 i_2} = a_{j_2 i_1} = 0
\end{array}
&
\begin{array}{c}
\xymatrix@R=1pc{
{i_1} \ar[rr] & & {j_2} \ar[dl] \\
& {k} \ar[ul] \ar[dr] \\
{j_1} \ar[ur] & & {i_2} \ar[ll]
}
\\
a_{i_1 j_2} \geq 1,\, a_{i_2 j_1} \geq 1 \\
a_{i_1 j_1} = a_{i_2 j_2} = 0
\end{array}
&
\begin{array}{c}
\xymatrix@=3pc{
{\cdot} \ar@{-}[r]^{i_1} \ar@{-}[d]^{j_1}
& {\cdot} \ar@{-}[d]^{j_2} \\
{\cdot} \ar@{-}[r]_{i_2} & {\cdot} \ar@{-}[ul]_k
}
\end{array}
\\ \hline
\end{array}
\]
\caption{Neighborhoods of $k$ when all sides are arcs.}
\label{t:nosegments}
\end{table}

\begin{remark}
In order to make the presentation concise, we have not explicitly
written down all the possible neighborhoods of $k$, but rather shown only
the arrows that must be present together with additional constraints
on numbers of arrows which are altogether enough to guarantee that
the mutation at $k$ does not change the total number of arrows.

Not every quiver satisfying these constraints is actually realized
as a neighborhood of a vertex $k$ in a quiver arising from a triangulation.
For example, we have not implied any restrictions on quantities like
$a_{i_1 i_2}$ and $a_{i_2 i_1}$ which are not affected by mutation at $k$
but are obviously bounded by $2$ for quivers arising from triangulations.
\end{remark}

\begin{remark}
The statement of Proposition~\ref{p:in1out1} does not hold for quivers
in general, not even for those arising from triangulations of marked
surfaces with punctures. For example, the left quiver below (an
orientation of the Dynkin diagram $D_4$) arises from a triangulation of
the disc with one puncture and four marked points on its
boundary~\cite[\S 6]{FST08}. The vertex $1$ has in-degree $2$, but the
corresponding mutation, given by the right quiver, has two more arrows.
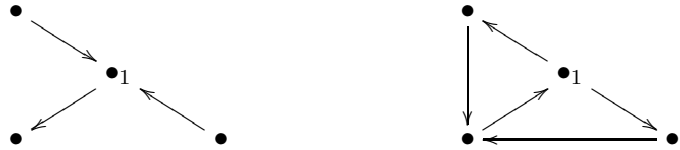
\begin{align*}
\xymatrix@R=1pc{
{\bullet} \ar[dr] \\
& {\bullet_1} \ar[dl] \\
{\bullet} & & {\bullet} \ar[ul]
}
& &
\xymatrix@R=1pc{
{\bullet} \ar[dd] \\
& {\bullet_1} \ar[ul] \ar[dr] \\
{\bullet} \ar[ur] & & {\bullet} \ar[ll]
}
\end{align*}
\end{remark}

\begin{lemma} \label{l:Qgb:in1out1}
Let $Q \in \cQ_{g,b}$ for $b>0$. Then there are no vertices of $Q$ with
in-degree $1$ and out-degree $1$.
\end{lemma}
\begin{proof}
Indeed, the existence of such vertex in a quiver arising from a
triangulation of an unpunctured bordered marked surface corresponds to
case~(2c) in Table~\ref{t:segments}, implying that there is a
boundary component containing at least two marked points.
\end{proof}

\begin{cor}
The quivers in a class $\cQ_{g,b}$ when $b>0$ have the same number
of arrows.
\end{cor}
\begin{proof}
Combine Proposition~\ref{p:in1out1} and Lemma~\ref{l:Qgb:in1out1}.
\end{proof}

\subsection{Quivers from once punctured closed surfaces}
\label{sec:Qg0}

We are left with the mutation classes $\cQ_{g,0}$ for $g>0$.

\begin{prop}
Let $Q$ be a quiver in $\cQ_{g,0}$ and let $k$ be a vertex of $Q$. Then
the neighborhood of $k$ is one of the four appearing in
Table~\ref{t:nosegments}. In particular, its in-degree and out-degree
are both $2$.
\end{prop}
\begin{proof}
From~\cite[Prop.~2.10]{FST08} we see that any triangulation of a
closed surface of genus $g>0$ with one puncture consists of $6g-3$ arcs.
Since $6g-3>2$ and all the arcs are incident to the single puncture, there
cannot be any self-folded triangles.

In the triangulation corresponding to $Q$, the arc corresponding to the
vertex $k$ is a side of two triangles whose other sides are the arcs
$i'$, $j'$ and $i''$, $j''$ as shown in the left picture
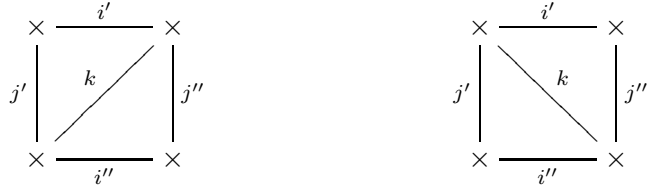
\begin{align*}
\xymatrix@=3pc{
{\times} \ar@{-}[r]^{i'} \ar@{-}[d]_{j'} & {\times} \ar@{-}[d]^{j''} \\
{\times} \ar@{-}[r]_{i''} \ar@{-}[ur]^k & {\times}
}
& &
\xymatrix@=3pc{
{\times} \ar@{-}[r]^{i'} \ar@{-}[d]_{j'} & {\times} \ar@{-}[d]^{j''} \\
{\times} \ar@{-}[r]_{i''} & {\times} \ar@{-}[ul]_k
}
\end{align*}
where $\times$ denotes the puncture.
The sets $\{i',i''\}$ and $\{j',j''\}$ are disjoint since otherwise
the triangulation or its flip at $k$ shown in the right picture would
contain self-folded triangles. Hence we are in one of the situations
depicted in Table~\ref{t:nosegments}, and we only need to verify the
constraints on the neighborhood of $k$ of the corresponding quivers.

Indeed, when $g=1$ (the once punctured torus) we get the quiver as in
case~(4a). Otherwise,  the fact that there are more than three arcs
incident to the puncture implies that the quivers and constraints
on their numbers of arrows are as shown in Table~\ref{t:nosegments}.

In case~(4c), for example, in the counterclockwise order around the
puncture, the arc $j_1$ does not immediately follow $i_1$ (top left),
$i_2$ does not follow $j_1$ (bottom left), $j_2$ does not follow $i_2$
(bottom right) and $i_1$ does not follow $j_2$ (top right), hence there
are arrows $j_1 \to i_1$ and $j_2 \to i_2$ coming from the triangles
containing $k$ but no arrows $j_1 \to i_2$ or $j_2 \to i_1$,
leading to the quiver with constraints as in Table~\ref{t:nosegments}.
\end{proof}

In particular we see that the number of arrows of a quiver in a class
$\cQ_{g,0}$ is twice the number of its vertices. This implies the
following corollary, thus completing the proof of
Theorem~\ref{t:const}.

\begin{cor}
Any quiver in $\cQ_{g,0}$ has $12g-6$ arrows.
\end{cor}

Next we provide explicit members from these classes in a similar way to
our treatment in~\cite[\S 3]{Ladkani11}.
Draw the fundamental polygon with $4g$ sides labeled
$1,2,1,2,\dots,2g-1,2g,2g-1,2g$ corresponding to a surface of genus
$g$, and identify the puncture with its vertices (recall that they are
all being identified on the surface). Any triangulation of the $4g$-gon
gives rise to a triangulation of the punctured (closed) surface of
genus $g$.

By taking triangulations as in Figure~\ref{fig:Tg0}, one gets explicit
such quivers which are shown in Figure~\ref{fig:Qg0}. They are built,
for $g>1$, by gluing two kinds of building blocks given in
Figure~\ref{fig:blocks}. Each block corresponds to a pair of
consecutive labels $2i-1, 2i$ of sides in the $4g$-gon and is obtained
by considering all the triangles incident to these sides. The left
block arises from the initial and terminal pairs of labels $\{1,2\}$
and $\{2g-1,2g\}$ whereas the right one arises from all other pairs
$\{2i-1,2i\}$.

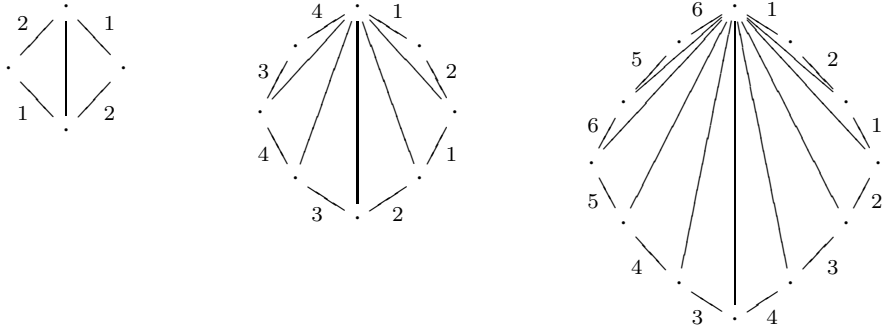
\begin{figure}
\begin{align*}
\xymatrix@=1pc{
& {\cdot} \ar@{-}[dr]^1 \ar@{-}[dd] \\
{\cdot} \ar@{-}[ur]^2 & & {\cdot} \ar@{-}[dl]^2 \\
& {\cdot} \ar@{-}[ul]^1
}
& &
\xymatrix@=0.3pc{
& && {\cdot}
\ar@{-}[drr]^1 \ar@{-}[dddrrr] \ar@{-}[dddddrr] \ar@{-}[dddddd]
\ar@{-}[dddlll] \ar@{-}[dddddll] \\
& {\cdot} \ar@{-}[urr]^4 && && {\cdot} \ar@{-}[ddr]^2 \\ \\
{\cdot} \ar@{-}[uur]^3 && && && {\cdot} \ar@{-}[ddl]^1 \\ \\
& {\cdot} \ar@{-}[uul]^4 && && {\cdot} \ar@{-}[dll]^2 \\
& && {\cdot} \ar@{-}[ull]^3
}
& &
\xymatrix@=0.2pc{
& && && {\cdot}
\ar@{-}[dddrrrr] \ar@{-}[dddddrrrrr] \ar@{-}[dddddddrrrr]
\ar@{-}[dddddddddrr] \ar@{-}[dddddddddd]
\ar@{-}[dddllll] \ar@{-}[dddddlllll] \ar@{-}[dddddddllll]
\ar@{-}[dddddddddll] \\
& && {\cdot} \ar@{-}[urr]^6 && && {\cdot} \ar@{-}[ull]_1 \\ \\
& {\cdot} \ar@{-}[uurr]^5 && && && && {\cdot} \ar@{-}[uull]_2 \\ \\
{\cdot} \ar@{-}[uur]^6 & && && && && & {\cdot} \ar@{-}[uul]_1 \\ \\
& {\cdot} \ar@{-}[uul]^5 && && && && {\cdot} \ar@{-}[uur]_2 \\ \\
& && {\cdot} \ar@{-}[uull]^4 && && {\cdot} \ar@{-}[uurr]_3 \\
& && && {\cdot} \ar@{-}[ull]^3 \ar@{-}[urr]_4
}
\end{align*}
\caption{Triangulations of closed surfaces of genus $g$ with one
puncture, for $g=1,2,3$. Arcs having the same label are identified.}
\label{fig:Tg0}
\end{figure}

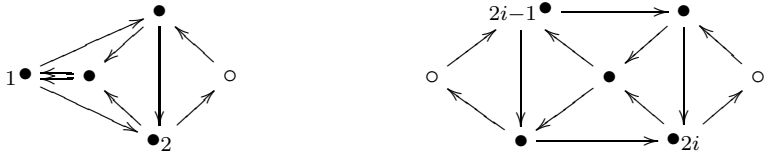
\begin{figure}
\begin{align*}
\xymatrix@=1pc{
& & {\bullet} \ar[dl] \ar[dd] \\
{_1\bullet} \ar[urr] \ar[drr] & {\bullet} \ar@/^0.1pc/[l] \ar@/_0.1pc/[l]
& & {\circ} \ar[ul] \\
& & {\bullet_2} \ar[ul] \ar[ur]
}
& &
\xymatrix@=1pc{
& {_{2i-1}\bullet} \ar[rr] \ar[dd] & & {\bullet} \ar[dl] \ar[dd] \\
{\circ} \ar[ur] & & {\bullet} \ar[ul] \ar[dl] & & {\circ} \ar[ul] \\
& {\bullet} \ar[ul] \ar[rr] & & {\bullet_{2i}} \ar[ul] \ar[ur]
}
\end{align*}
\caption{Two building blocks for quivers in $\cQ_{g,0}$ ($g>1$) arising
from triangulations as in Figure~\ref{fig:Tg0}. Gluing points are
marked with $\circ$.} \label{fig:blocks}
\end{figure}

The close connection between triangulations of the $4g$-gon and quivers
in $\cQ_{g,0}$ can be expressed more precisely in terms of mutation
classes of the Dynkin diagrams $A_n$. Let $n \geq 1$ and denote by
$A_n$ the quiver with $n$ vertices labeled $1,2,\dots,n$ and arrows $i
\to i+1$ for $1 \leq i < n$. The quivers in the mutation class of $A_n$
are those arising from triangulations of the $(n+3)$-gon (disc with no
punctures and $n+3$ marked points on its boundary), see~\cite{CCS06}.
They have been explicitly described in~\cite{BuanVatne08}.

\begin{prop}
Let $g,n>0$. Then a quiver in the mutation class of $A_n$ is
a full subquiver of some quiver in the class $\cQ_{g,0}$ if
and only if $n \leq 4g-3$.
\end{prop}
\begin{proof}
We show first that any quiver $Q'$ in the mutation class of $A_{4g-3}$
is a full subquiver of some quiver in $\cQ_{g,0}$. Indeed, by taking
the fundamental $4g$-gon of a genus $g$ surface we see that any
triangulation of the $4g$-gon corresponding to $Q'$ yields a
triangulation of the closed genus $g$ surface with one puncture, giving
a quiver $Q \in \cQ_{g,0}$ containing $Q'$ as a full subquiver. For
example, for each of the quivers in Figure~\ref{fig:Qg0}, its full
subquiver on the vertices not corresponding to the sides of the
$4g$-gon is the quiver $A_{4g-3}$.

Conversely, we show that no quiver in the mutation class of $A_{4g-2}$
is a full subquiver of a quiver in $\cQ_{g,0}$. Observe that if $Q'$ is
a full subquiver of $Q$, then by performing a mutation at a vertex $k$
of $Q'$ we see that $\mu_k(Q')$ is a full subquiver of $\mu_k(Q)$.
Hence it is enough to prove the claim for the quiver $A_{4g-2}$ itself.

Assume to the contrary that $A_{4g-2}$ is a full subquiver on the
vertices $1,2,\dots,4g-2$ of some quiver $Q \in \cQ_{g,0}$. Since each
vertex of $Q$ has two incoming and two outgoing arrows, we deduce that
$Q$ must contain the following arrows:
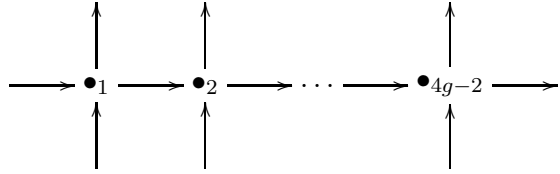
\[
\xymatrix{
& & & & \\
\ar[r] & {\bullet_1} \ar[r] \ar[u] & {\bullet_2} \ar[r] \ar[u]
& {\ldots} \ar[r] & {\bullet_{4g-2}} \ar[r] \ar[u] & \\
& \ar[u] & \ar[u] & & \ar[u]
}
\]
Counting them, we get $3(4g-2)+1 = 12g-5$, contradicting the fact that
$Q$ has only $12g-6$ arrows.
\end{proof}

\section{Interpretation via derived equivalences}
\label{sec:dereq}

In this section we explain how to interpret, in certain cases,
mutations of quivers with potentials preserving the number of arrows as
derived equivalences of the corresponding Jacobian algebras. This
interpretation which naturally holds for reflections, is valid also for
mutations of quivers arising from triangulations of unpunctured
surfaces as well as for the quivers in the classes $\cQ_{g,0}$.

Throughout this section, we fix a field $K$. Recall that two
$K$-algebras are called \emph{derived equivalent} if their module
categories have equivalent derived categories. Derived equivalent
algebras share many homological properties, and we refer the reader to
the survey~\cite{Keller07} for further details on tilting theory and
derived equivalence.

Recall that the \emph{path algebra} $KQ$ of a quiver $Q$ has a basis
consisting of the paths in $Q$, and the product of any two paths is
their concatenation, if defined, and zero otherwise. A \emph{quiver
with potential} (QP) is a pair $(Q,W)$ where $Q$ is a quiver (under our
assumption throughout the paper, it does not have loops and $2$-cycles)
and $W$ is a \emph{potential}, which we assume to be polynomial, i.e.\
a linear combination of cycles in $KQ$. In~\cite{DWZ08} the authors
have defined the notion of mutation of QP at a vertex. Under certain
conditions, it extends the notion of mutation of quivers. The
(non-completed) \emph{Jacobian algebra} of $(Q,W)$, denoted by
$\cP(Q,W)$, is the quotient of $KQ$ by the ideal generated by the
cyclic derivatives of $W$ with respect to all the arrows,
see~\cite{DWZ08}.

\subsection{Reflections}

We have already mentioned that \emph{reflection}, that is, mutation at
a sink or a source, preserves the number of arrows in the quiver.
Reflection has also a well-known representation theoretic
interpretation as derived equivalence. Namely, if $(Q,W)$ is any QP and
$k$ is a sink or a source of $Q$, then the Jacobian algebras $\cP(Q,W)$
and $\cP(\mu_k(Q,W))$ are derived equivalent. Indeed, one takes the
left complex (when $k$ is sink) or the right one (when it is a source)
of finitely generated right projective $\cP(Q,W)$-modules
\begin{align}
\tag{$\star$} \label{e:APR} \bigl( P_k \xrightarrow{f} \bigoplus_{j \to
k} P_j \bigr) \oplus \bigl(\bigoplus_{i \neq k} P_i \bigr) & & \bigl(
\bigoplus_{k \to j} P_j \xrightarrow{g} P_k \bigr) \oplus
\bigl(\bigoplus_{i \neq k} P_i \bigr)
\end{align}
where $P_i$ denotes the projective module corresponding to $i$ spanned
by all paths starting at $i$, the map $f$ (respectively, $g$) is
induced by all the arrows ending (respectively, starting) at $k$, and
the terms $P_i$ for $i \neq k$ lie in degree $0$. This is a tilting
complex over $\cP(Q,W)$ whose endomorphism algebra is isomorphic to
$\cP(\mu_k(Q,W))$, so by Rickard's theorem~\cite{Rickard89} the two
algebras are derived equivalent. In the finite-dimensional case this is
an instance of APR-tilt~\cite{APR79} generalizing the BGP
reflections~\cite{BGP73} between path algebras of quivers without
oriented cycles.

\subsection{Quivers from unpunctured surfaces}

For quivers arising from triangulations of marked bordered surfaces,
potentials have been defined by Labardini-Fragoso~\cite{Labardini09} in
such a way that flips of triangulations correspond to mutations of the
associated QP. When the surface has no punctures, these potentials are
sums of the oriented $3$-cycles in the quiver which are induced by the
internal triangles of the triangulation, and the corresponding Jacobian
algebras are the finite-dimensional gentle algebras introduced by
Assem, Br\"{u}stle, Charbonneau-Jodoin and Plamondon~\cite{ABCP10}.

In general, the number of arrows in the quiver of any gentle algebra is
invariant under derived equivalence~\cite{AvellaAlaminosGeiss08}, but
for the gentle algebras arising from triangulations we can actually say
more; by combining Proposition~\ref{p:in1out1} and our previous
work~\cite[\S 2]{Ladkani11} we deduce that for any single mutation of a
QP arising from a triangulation of a marked unpunctured surface, the
condition that the number of arrows is preserved is equivalent to the
derived equivalence of the Jacobian algebras.

\begin{theorem} \label{t:goodmut}
Let $(Q,W)$ be a QP arising from a triangulation of marked bordered
surface without punctures. For any vertex $k$ of $Q$, the following
conditions are equivalent:
\begin{enumerate}
\renewcommand{\theenumi}{\roman{enumi}}
\item
The numbers of incoming and outgoing arrows at $k$ are not both equal
to $1$;

\item
$Q$ and its mutation $\mu_k(Q)$ have the same number of arrows;

\item
The algebras $\cP(Q,W)$ and $\cP(\mu_k(Q,W))$ are derived equivalent.
\end{enumerate}
\end{theorem}

In particular, using Theorem~\ref{t:const} we get another proof of the
following result.

\begin{cor}[\cite{Ladkani11}]
Each of the classes $\cQ_{g,b}$ when $b>0$ can be regarded as mutation
class of QP whose Jacobian algebras are all derived equivalent.
\end{cor}

\begin{remark}
Combining Theorem~\ref{t:goodmut} with the recent computation by
David-Roesler and Schiffler in~\cite{DavidRoeslerSchiffler11} of the
Avella-Alaminus-Geiss derived invariants~\cite{AvellaAlaminosGeiss08}
allows for a description of the derived equivalence classes of the
Jacobian algebras of QP arising from triangulations of marked
unpunctured surfaces in terms of the properties of the corresponding
triangulations, generalizing the derived equivalence classifications of
cluster-tilted algebras of Dynkin type $A$~\cite{BuanVatne08} and
affine type $\tilde{A}$~\cite{Bastian09}. This is a subject of further
investigations.
\end{remark}

\begin{remark}
The statement of Theorem~\ref{t:goodmut} does not hold for QP in
general, not even for those arising from triangulations of marked
surfaces with some punctures. For example, the following picture shows
certain quivers arising from triangulations of the once-punctured
pentagon (i.e.\ a disc with one puncture and five marked points on its
boundary). The corresponding Jacobian algebras are cluster-tilted
algebras of type $D_5$~\cite{Schiffler08}.
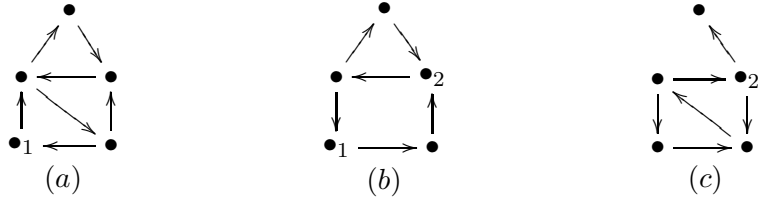
\begin{align*}
\begin{array}{c}
\xymatrix@=0.3pc{
& {\bullet} \ar[ddr] \\ \\
{\bullet} \ar[ddrr] \ar[uur] & & {\bullet} \ar[ll] \\ \\
{\bullet_1} \ar[uu] & & {\bullet} \ar[ll] \ar[uu]
}
\\
(a)
\end{array}
& &
\begin{array}{c}
\xymatrix@=0.3pc{
& {\bullet} \ar[ddr] \\ \\
{\bullet} \ar[uur] \ar[dd] & & {\bullet_2} \ar[ll] \\ \\
{\bullet_1} \ar[rr] & & {\bullet} \ar[uu]
}
\\
(b)
\end{array}
& &
\begin{array}{c}
\xymatrix@=0.3pc{
& {\bullet} \\ \\
{\bullet} \ar[rr] \ar[dd] & & {\bullet_2} \ar[uul] \ar[dd] \\ \\
{\bullet} \ar[rr] & & {\bullet} \ar[uull]
}
\\
(c)
\end{array}
\end{align*}

The quivers in~(a) and~(b) are related by mutation at the vertex $1$.
They have different numbers of arrows but the corresponding algebras
are derived equivalent. The quivers in~(b) and~(c) are related by
mutation at the vertex $2$. They have the same number of arrows but the
corresponding algebras are not derived equivalent, see~\cite{BHL10}.
\end{remark}

\begin{remark}
Furthermore, there are mutation classes consisting of finitely many QP
whose Jacobian algebras are finite-dimensional and derived equivalent,
but the quivers themselves have varying numbers of arrows. Examples are
the mutation classes of the exceptional quivers $E^{(1,1)}_6$ and
$X_6$, see~\cite{Ladkani11}.
\end{remark}

\subsection{The classes $\cQ_{g,0}$}

For a quiver $Q$ in the classes $\cQ_{g,0}$ arising from triangulations
of closed surfaces with one puncture, the potential defined
in~\cite{Labardini09} is the sum of two terms $W_\Delta + x_p W_p$ for
some $0 \neq x_p \in K$. The term $W_\Delta$ is the sum of the oriented
$3$-cycles corresponding to the triangles comprising the triangulation,
just as in the unpunctured case. The term $W_p$ is the oriented cycle
induced by traversing the arcs of the triangulation in a
counter-clockwise order at the puncture $p$. Thus, any arrow in $Q$
appears in the cycle $W_p$ exactly once.

Since there are never self-folded triangles in a triangulation of a
once punctured closed surface (cf.\ Section~\ref{sec:Qg0}), the
argument in~\cite{Labardini09} showing that flips of triangulations
correspond to mutations of the associated QP is applicable also when we
set $x_P=0$, see Cases~1 and~2 in the proof
of~\cite[Theorem~30]{Labardini09}. Hence, the association of the
potential $W=W_\Delta$ to the quiver $Q$ is compatible with quiver
mutations, a fact which can be also verified directly by considering
the local neighborhoods in Table~\ref{t:nosegments}. The potential $W$
is thus non-degenerate but not rigid. In the special case of $g=1$, we
recover Example~8.6 of~\cite{DWZ08}. The Jacobian algebra of $(Q,W)$
satisfies all the conditions in the definition of a gentle algebra,
except that it is \emph{infinite}-dimensional.

\begin{prop}
Let $Q \in \cQ_{g,0}$ for some $g>0$ and let $W$ be the associated
potential. Then, for any vertex $k$ of $Q$ the two complexes given
by~\eqref{e:APR} are tilting complexes over $\cP(Q,W)$ and their
endomorphism algebras are both isomorphic to $\cP(\mu_k(Q,W))$.
\end{prop}
\begin{proof}
For $g=1$, this can be checked directly. For $g>1$ one considers the
situation ``locally'' at the neighborhood of $k$ and reduces to the
finite-dimensional case which was treated in~\cite[\S 2]{Ladkani11}.
\end{proof}

\begin{remark}
The proposition implies that each class $\cQ_{g,0}$ with the above
associated potentials satisfies a stronger version of condition
$(\delta_3)$ that was stated in~\cite{Ladkani11}. In particular,
$\cQ_{g,0}$ can be regarded as mutation class of QP whose Jacobian
algebras are all derived equivalent.
\end{remark}

\begin{remark}
The statement of the proposition does not hold for the classes
$\cQ_{g,b}$ when $b>0$ despite the fact that all the Jacobian algebras
are derived equivalent. Indeed, in any quiver $Q \in \cQ_{g,b}$ with
$b>0$ there is at least one vertex $k$ whose in-degree and out-degree
are not both equal to $2$, and then one of the complexes
in~\eqref{e:APR} is not a tilting complex.
\end{remark}

\begin{remark}
By work of Keller and Yang~\cite[\S 6]{KellerYang11}, the statement of
the proposition holds for any quiver with potential $(Q,W)$ whose
Ginzburg dg-algebra has its cohomology concentrated in degree zero.
This assumption implies that the Jacobian algebra $\cP(Q,W)$ is
$3$-Calabi-Yau.

However, the Jacobian algebra of a quiver $Q \in \cQ_{g,0}$ with the
above associated potential is \emph{not} $3$-Calabi-Yau. Indeed, since
the potential $W$ we associate to $Q$ is a sum of cycles of the same
length, the Jacobian algebra $\cP(Q,W)$ is naturally graded. One can
thus check that it is not $3$-Calabi-Yau by computing its matrix
Hilbert series and using~\cite[Theorem~4.6]{Bocklandt08}.
\end{remark}

\bibliographystyle{amsplain}
\bibliography{arrows}

\end{document}